\documentclass[11pt]{amsart}

\usepackage[table]{xcolor}

\usepackage{ytableau}
\usepackage{tikz,varwidth}
\usetikzlibrary{calc, cd}
\usepackage[enableskew]{youngtab}
\usepackage[top=1in,left=1in,right=1in,bottom=1in]{geometry}
\usepackage{enumerate, amssymb, array}
\usepackage[all,cmtip]{xy}

\newtheorem{thm}{Theorem}[section]
\newtheorem{prop}[thm]{Proposition}
\newtheorem{claim}[thm]{Claim}

\newtheorem{algorithm}{Algorithm}

\theoremstyle{definition}

\newtheorem{convention}[thm]{Convention}

\newcommand{\ZZ}{\mathbb{Z}}

\newcommand{\NN}{\mathbb{N}}

\newcommand{\comment}[1]{}

\newcommand{\Semi}{\mathrm{SS}}
\newcommand{\bfc}{\mathbf{c}}
\newcommand{\bfe}{\mathbf{e}}
\newcommand{\bfw}{\mathbf{w}}
\newcommand{\bfx}{\mathbf{x}}
\newcommand{\bfzr}{\mathbf{0}}
\newcommand{\wt}{\widetilde}

\newcommand{\blue}[1]{\textcolor{blue}{#1}}
\newcommand{\red}[1]{\textcolor{red}{#1}}

\newcommand{\Grdab}{G^{r,\alpha,\beta}_d}

\newcommand{\hide}[1]{}

\newtheorem{Definition}[thm]{Definition}
\newenvironment{definition}
  {\begin{Definition}\rm}{\end{Definition}}

\newtheorem{Example}[thm]{Example}
\newenvironment{example}
  {\begin{Example}\rm}{\end{Example}}

\newtheorem{Fact}[thm]{Fact}
\newenvironment{fact}
  {\begin{Fact}\rm}{\end{Fact}}

\newtheorem{Theorem}[thm]{Theorem}
\newenvironment{theorem}
  {\begin{Theorem}\rm}{\end{Theorem}}
  
\newtheorem{Lemma}[thm]{Lemma}
\newenvironment{lemma}
  {\begin{Lemma}\rm}{\end{Lemma}}

\newtheorem{Remark}[thm]{Remark}
\newenvironment{remark}
  {\begin{Remark}\rm}{\end{Remark}}

\newtheorem{Proposition}[thm]{Proposition}
\newenvironment{proposition}
  {\begin{Proposition}\rm}{\end{Proposition}}

\newtheorem{Corollary}[thm]{Corollary}
\newenvironment{corollary}
  {\begin{Corollary}\rm}{\end{Corollary}}

\theoremstyle{remark}

\newcommand \defnow[1]{\begin{definition}{#1}\end{definition}}
\newcommand \exnow[1]{\begin{example}{#1}\end{example}}
\newcommand \lemnow[1]{\begin{lemma}{#1}\end{lemma}}

\newcommand \remnow[1]{\begin{remark}{#1}\end{remark}}

\newcommand \enumnow[1]{\begin{enumerate}{#1}\end{enumerate}}
\newcommand \itemnow[1]{\begin{itemize}{#1}\end{itemize}}

\title{Combinatorial relations on skew Schur and skew stable Grothendieck polynomials}
\author[M. Chan]{Melody Chan}\address{Department of Mathematics, Brown University, Box
1917, Providence, RI 02912}\email{melody\_chan@brown.edu}

\author[N. Pflueger]{Nathan Pflueger}\address{Department of Mathematics and Statistics, Amherst College, Amherst, MA 01002}\email{npflueger@amherst.edu}
\date{\today}

\begin{document}

\begin{abstract}
We give a combinatorial expansion of the stable Grothendieck polynomials of skew Young diagrams in terms of skew Schur functions, using a new row insertion algorithm for set-valued semistandard tableaux of skew shape. This expansion unifies some previous results: it generalizes a combinatorial formula obtained in earlier joint work with L\'opez Mart\'in and Teixidor i Bigas concerning Brill-Noether curves, and it generalizes a 2000 formula of Lenart and a recent result of Reiner-Tenner-Yong to skew shapes. We also give an expansion in the other direction: expressing skew Schur functions in terms of skew Grothendieck polynomials.

\bigskip

\end{abstract}

\maketitle

\tableofcontents

\section{Introduction}

%The motivation of this paper is a natural question concerning the enumeration of set-valued semistandard tableaux of skew shape.  Namely, 
%Given a skew shape $\sigma$ and a vector $\bfc = (c_1,c_2,\ldots)$ of nonnegative integers, how many semistandard set-valued tableaux of shape $\sigma$ and content $\bfc$ are there?  

Let $\sigma$ be a skew Young diagram (Definition~\ref{def:young}).
The main result of this paper is a new formula for the {\em skew stable Grothendieck polynomial} $G_\sigma$ of Lascoux-Sch\"utzenberger and Fomin-Kirillov  \cite{lascoux-schutzenberger,fomin-kirillov-grothendieck} as a linear combination of skew Schur functions $s_\lambda$ on related shapes $\lambda$. As was demonstrated by Buch, the coefficients of $G_\sigma$  have a combinatorial interpretation in terms of set-valued tableaux \cite{buch-littlewood-richardson}, and
our original motivation for this paper came from a recent geometric result, proved in a companion paper \cite{chan-pflueger-euler}, identifying Euler characteristics of Brill-Noether varieties up to sign as counts of set-valued standard tableaux. %, which we discuss below.

%Since computationally efficient formulas for these coefficients are hard to come by, 
It is natural to ask for a linear expansion of $G_\sigma$ in terms of other symmetric functions, particularly the basis of Schur functions. 
Such an expansion was obtained by Fomin-Greene, who in fact obtained such expansions for a wide class of symmetric functions including stable Grothendieck polynomials associated to arbitrary permutations \cite{fomin-greene-noncommutative}.  (Note that the stable Grothendieck polynomials of $321$-avoiding permutations precisely correspond to stable Grothendieck polynomials of skew shapes as in \cite{buch-littlewood-richardson}, by a theorem of Billey-Jockusch-Stanley \cite{billey-jockusch-stanley-some}.)   Buch's expansion of skew Grothendieck polynomials in terms of Grothendieck polynomials of {\em straight} shapes, along with Lenart's expansion of the latter into Schur functions, provides another route to such an expansion \cite{buch-littlewood-richardson,lenart-combinatorial}.  

Our main theorem expresses $G_\sigma$ instead as a linear combination of {\em skew} Schur functions $s_\lambda$.  
The coefficients of the linear combination have explicit combinatorial interpretations which we provide; they count appropriate auxiliary tableaux.  We state the result below, postponing all definitions to the next section.

\begin{thm}\label{thm:comb}
For any connected skew shape $\sigma$, the skew Grothendieck polynomial $G_\sigma$ admits a linear expansion $$G_\sigma = \sum_{\mu\supseteq \sigma} (-1)^{|B(\mu/\sigma)|} a_{\sigma,\mu} \cdot s_\mu$$
where the $a_{\sigma, \mu}$ are nonnegative integers and the $s_\mu$ are skew Schur functions. Here $a_{\sigma,\mu}$ is the product of the following two integers:
\enumnow{
\item the number of row-weakly-bounded semistandard tableaux of shape $A(\mu/\sigma)$, and
\item the number of row-bounded, reverse row-strict tableaux of shape $B(\mu/\sigma)$.
}
\end{thm}

\noindent Here $A(\mu/\sigma)$ and $B(\mu/\sigma)$ are the subshapes of $\mu$ lying above and below $\sigma$, respectively.  In fact, this statement is a specialization of a more general formula for row-refined skew stable Grothendieck polynomials that we obtain in Theorem~\ref{t:refinedcount}.  

Theorem~\ref{thm:comb} generalizes Lenart's theorem from 2000 expanding Grothendieck polynomials for non-skew shapes into non-skew Schur functions \cite{lenart-combinatorial}; in fact, that result is a visible specialization.  We explain this connection in detail in Remark~\ref{rem:connection-to-lenart}.  We also give a theorem in the other direction, Theorem~\ref{t:back}, expressing $s_\sigma$ in terms of polynomials $G_\mu$, for skew shapes $\sigma$.  This generalizes an analogous theorem of Lenart from the same paper.

To be clear, skew Schur functions, since they include Schur functions properly, are evidently not a basis for the space of symmetric functions; thus the coefficients of our expansion are not canonical. 
To provide a point of comparison, a result in the literature that is similar in spirit to Theorem~\ref{thm:comb} is the skew Pieri rule of Assaf-McNamara, in which the product of a skew shape and a rectangle is expressed in terms of other skew shapes \cite[Theorem 3.2]{assaf-mcnamara-pieri}. Again, this expression is necessarily noncanonical, but it is combinatorially natural using an insertion algorithm.  Our proof also uses a new insertion algorithm for skew set-valued semistandard tableaux that is related to previous work of Bandlow-Morse, and indeed our algorithm may be interpreted as extending to the skew case some of their results \cite[\S5]{bandlow-morse-combinatorial}.  In fact, it recalls earlier work of Sagan-Stanley row insertion for skew (non-set-valued) tableaux \cite{sagan-stanley-robinson}, as well as Buch's ``uncrowding'' algorithm on set-valued tableaux \cite[\S 6]{buch-littlewood-richardson}.  We also note that using insertion operations to derive such combinatorial identities has been carried out previously, in the form of Hecke insertion operations studied in \cite{bksty-stable}.

\medskip

{\bf Motivation from geometry.} 
Our original motivation came from a recent result in Brill-Noether theory that we prove in a pair of companion papers (\cite{chan-pflueger-euler}, together with \cite{chan-pflueger-relative} which proves an auxiliary result).  %:\footnote{The preprint \cite{chan-pflueger-euler} is being split into an algebro-geometric paper and this paper, which together subsume it.}
\begin{thm}\cite{chan-pflueger-euler} Fix $r,d\ge 0$ and nondecreasing sequences $\alpha,\beta\in\ZZ^{r+1}_{\ge 0}$.  Let $(X,p,q)$ be a general twice-pointed curve of genus $g$ over an algebraically closed field.
Then the algebraic Euler characteristic of the Brill-Noether variety $\Grdab(X,p,q)$ is 
$$\chi(\Grdab(X,p,q)) = (-1)^{g-|\sigma|} \cdot \#(\text{standard set-valued tableaux on $\sigma$ of content }\{1,\ldots,g\}).$$
\end{thm} 
\noindent Here $\sigma$ is the skew-shape obtained from an $(r+1)\times (g-d+r)$ rectangle by adding $\alpha_r \le \cdots \le \alpha_0$ boxes down the left side and $\beta_0\ge \cdots \ge \beta_r$ boxes down the right side.  The partitions $\alpha$ and $\beta$ encode some ramification conditions imposed at the two marked points $p,q$ of $X$.  Roughly speaking, from the geometric perspective it is natural to seek formulae for set-valued tableaux in terms of skew Schur functions, which do not give preference to one marked point over the other, rather than (straight) Schur functions, which do.
Indeed, 
our result provides a combinatorial explanation of the main theorem of \cite{anderson-chen-tarasca-k-classes}, as we shall explain further in Remark~\ref{rem:connection-to-act}.    It also generalizes a theorem with L\'opez and Teixidor i Bigas which computes genera of Brill-Noether curves  \cite{clpt}. (That case corresponds to the situation in which there is exactly one more label than the number of boxes.)  
In addition, a recent result of Reiner-Tenner-Yong  
is also a special case of Theorem~\ref{thm:comb}, and in fact, their work inspired some of the results here \cite[Corollary 3.11]{reiner-tenner-yong}. 
%%%%%%%%%%%%%%%%%%%%%%%%%%%%%%%%%%%%%%%%%%%%%%%%%%%%
%%%%%%%%%%%%%%%%%%%%%%%%%%%%%%%%%%%%%%%%%%%%%%%%%%%%

\section{Preliminaries}\label{sec:prelim}

We now give some preliminaries on tableaux.  First we define a skew Young diagram: this is almost the same as the usual definition, except that we only record the set of boxes in a diagram rather than remembering a formal difference of two partitions.
Fix the partial order $\preceq$ on $\ZZ^2$ given by $(x,y)\preceq (x',y')$ if $x\le x'$ and $y\le y'$.  

\begin{definition}\label{def:young}\mbox{}
\begin{enumerate}
\item 
A skew Young diagram is a finite subset $\sigma \subset \ZZ_{>0}^2$ that is closed under taking intervals.  In other words, $\sigma$ has the property that if $(x,y)$ and $(x',y')\in  \sigma$ with $(x,y)\preceq (x',y')$, then $$\{(x'',y''):(x,y)\preceq (x'',y'') \preceq (x',y')\}\subseteq \sigma.$$
\item A skew Young diagram is called a {\em Young diagram} if $\sigma$ is empty or has a unique minimal element.
\end{enumerate}
\end{definition}

Skew Young diagrams are sometimes also called {\em skew shapes}, and skew Young diagrams having a unique minimal element will sometimes be called {\em straight shapes} for emphasis. 
In accordance with the English notation for Young diagrams, we will draw the points of $\ZZ^2$ arranged with $x$-coordinate increasing from left to right, and $y$-coordinate {\em increasing} from top to bottom, e.g.
$$
\begin{array}{ccc}
(1,1) & (1,2) & \cdots \\
(2,1) & (2,2) & \\
\vdots & &
\end{array}
$$
Furthermore, we will draw, and refer to, the members of $\sigma$ as boxes, as usual, and we let $|\sigma|$ denote the number of boxes in $\sigma$.  We shall assume throughout for convenience that $\sigma$ is a connected shape, i.e.~its Hasse diagram is connected (see Remark \ref{rem:disconn} on the disconnected case). %We do not lose generality in doing so, since our main object of study, the Grothendieck polynomial, can be easily recovered from its connected components of an arbitrary shape.  %If instead $\sigma$ has two disconnected pieces $\sigma_1$ and $\sigma_2$, then its Grothendieck polynomial, Definition~\ref{def:schur-groth} below, is  $G_\sigma = G_{\sigma_1} G_{\sigma_2}$.

\begin{definition} \label{def:tableau}
A {\em tableau} of shape $\sigma$ is an assignment $T$ of a positive integer, called a {\em label}, to each box of $\sigma$.  \begin{enumerate}
\item A tableau $T$ of shape $\sigma$ is {\em semistandard} if the rows of $\sigma$ are weakly increasing from left to right, and the columns of $\sigma$ are strictly increasing from top to bottom.
\item A tableau $T$ of shape $\sigma$ is {\em standard} if it is semistandard and furthermore each integer $1,\ldots,|\sigma|$ occurs exactly once as a label.
\end{enumerate}

\end{definition}

\begin{definition}\cite{buch-littlewood-richardson} \label{def:svt}
A {\em set-valued tableau} of shape $\sigma$ is an assignment of a nonempty finite set of positive integers  to each box of $\sigma$. 

Given sets $S,T\subseteq \ZZ_{>0}$, we write $S<T$ if $\max(S)<\min(T)$, and we write $S\le T$ if  $\max(S)\le\min(T)$.  Then we extend the definitions of {\em semistandard} and {\em standard} tableaux to set-valued tableaux.
\begin{enumerate}
\item \label{it:sstd}A set-valued tableau $T$ of shape $\sigma$ is {\em semistandard} if the rows of $\sigma$ are weakly increasing from left to right, and the columns of $\sigma$ are strictly increasing from top to bottom.
\item \label{it:std} A set-valued tableau $T$ of shape $\sigma$ is {\em standard} if it is semistandard and furthermore the labels are pairwise disjoint sets with union $\{1,\ldots,r\}$ for some $r \ge |\sigma|$.
\end{enumerate}
Denote by $\Semi(\sigma)$ the set of all semistandard set-valued tableaux on $\sigma$.
\end{definition}

%There is an evident bijection between tableaux of shape $\sigma$ and set-valued tableaux of shape $\sigma$ in which each label is a one element set.  We will identify these objects without further mention.

Let $\bfc=(c_1,c_2,\ldots)$ be a nonnegative integer sequence that is eventually zero.  We say that a tableau or set-valued tableau $T$ of shape $\sigma$ has content $\bfc = c(T)$ if label $i$ appears exactly $c_i$ times, for all $i$.  
Write $|T| = |c(T)| = \sum c_i$ for the total number of labels.

\begin{definition} \label{def:schur-groth} \mbox{}
\begin{enumerate}
\item
For any skew shape $\sigma$, the {\em skew Schur function} $s_\sigma$ is
$$s_\sigma =\sum_T {\mathbf{x}}^{c(T)}$$
as $T$ ranges over all semistandard fillings of $\sigma.$
\item
For any skew shape $\sigma$, the {\em skew stable Grothendieck polynomial} $G_\sigma$ is
$$G_\sigma =\sum_T (-1)^{|T|-|\sigma|}{\mathbf{x}}^{c(T)}$$
as $T$ ranges over all semistandard set-valued fillings of $\sigma.$
\end{enumerate}
\end{definition}

Given a set-valued tableau $T$ of shape $\sigma$, define the {\em excess} of $T$, denoted $e(T)$, as the vector $\bfe =  (e_1,e_2,\ldots)$ in which $e_i$ records the number of labels in row $i$ in excess of the number of boxes in row $i$.  Therefore $|\sigma|+|e(T)|=|c(T)|$.  

\begin{remark} \label{rem:disconn}
Definition \ref {def:schur-groth} does not require $\sigma$ to be a connected skew shape, but there is little loss of generality in focusing on the connected case, as we do in this paper. If $\sigma$ is a union of two disconnected parts $\sigma_1, \sigma_2$, then a filling $T$ of $\sigma$ is semistandard if and only if the resulting fillings $T_1$ of $\sigma_1$ and $T_2$ of $\sigma_2$ are semistandard, since no box in $\sigma_1$ is comparable by $\preceq$ to a box of $\sigma_2$. Also, $c(T) = c(T_1) + c(T_2)$. It follows that $G_\sigma = G_{\sigma_1} G_{\sigma_2}$. Hence Grothendieck polynomials of disconnected skew shapes factor into those of connected skew shapes.
\end{remark}

\section{Row insertion for skew set-valued tableaux}

We now set notation for a refinement of the Grothendieck polynomial based on the excess statistic, and we prove a theorem expressing it linearly in terms of skew Schur functions. Note that the idea of introducing an additional parameter into the Grothendieck polynomial goes back already to \cite{fomin-kirillov-grothendieck}.

\begin{definition} \label{def:rg} Let $\sigma$ be a skew Young diagram.
We define the {\em row-refined skew stable Grothendieck polynomial} of $\sigma$ to be the power series 
$$RG_\sigma(\mathbf{x};\mathbf{w}) = \sum_{T\in \Semi(\sigma)} (-1)^{|e(T)|} \mathbf{x}^{c(T)} \mathbf{w}^{e(T)}.$$
\end{definition}
\noindent Thus $RG_\sigma(\mathbf{x};\mathbf{1}) = G_\sigma(\mathbf{x})$, so the usual skew stable Grothendieck polynomial is obtained as a specialization.

\begin{definition}Let $\mu$ be a skew Young diagram.
\begin{enumerate}
\item A tableau $T$ of shape $\mu$ is {\em reverse row-strict} if its rows are strictly decreasing from left to right, and its columns are weakly decreasing from top to bottom. 
\item A tableau $T$ of shape $\mu$ is {\em row-bounded}, respectively {\em row weakly-bounded}, if for every box $(i,j)$ in $\mu$, $T(i,j) < i$, respectively $T(i,j)\le i$.
\end{enumerate}
\end{definition}

We henceforth adopt the following convention governing containment of Young diagrams. 
\begin{convention}\label{conv}
Fix $\sigma$ a skew shape; we take the numbering of the rows of $\sigma$ to start at $1$ at the top.  For another skew shape $\lambda$, we write $\lambda \supseteq \sigma$ if every box of $\sigma$ is a box of $\lambda$, {\em and furthermore every box of $\lambda$ is in the same column as some box of $\sigma$}. In other words, we will only consider skew shapes $\lambda\supseteq \sigma$ that occupy the same set of columns as $\sigma$. They are not allowed to extend $\sigma$ to the right or to the left.
\end{convention}

By Convention~\ref{conv}, if $\sigma$ is a connected skew shape and $\lambda\supseteq \sigma$, then $\lambda - \sigma$ consists of a set of boxes above $\sigma$ and a set of boxes below $\sigma$.  Write $A(\lambda/\sigma)$ and $B(\lambda/\sigma)$ for these respective skew Young diagrams; $A$ and $B$ stand for {\em above} and {\em below.}  We emphasize that, contrary to some conventions, $\lambda$ may extend $\sigma$ both above and below.

\begin{theorem}\label{t:refinedcount} For any connected skew shape $\sigma$, 
$$RG_\sigma(\bfx;\bfw) = \sum_{(\mu,\bfe)} (-1)^{| B(\mu/\sigma)|} \,a_{\sigma,\mu,\bfe} \cdot s_\mu(\bfx) \cdot\bfw^\bfe $$
where the sum is over all skew shapes $\mu\supseteq \sigma$ and sequences $\bfe$, and the numbers $a_{\sigma,\mu,\bfe}$ are nonnegative integers.  Specifically, $a_{\sigma,\mu,\bfe}$ is the number of pairs $(T',T'')$ such that
\begin{itemize}
\item $T'$ is a row-weakly bounded semistandard tableau on $A(\mu/\sigma)$, and
\item $T''$ is a reverse-row-strict, row-bounded tableau on $B(\mu/\sigma)$, 
\end{itemize}
satisfying $$c(T') + c(T'') = \bfe.$$
\end{theorem}

For convenience, we record the coefficient-by-coefficient interpretation of Theorem~\ref{t:refinedcount}.  Let $\Semi_{\bfc,\bfe}(\sigma)$ denote the set of semistandard set-valued fillings of $\sigma$ of content $\bfc$ and excess $\bfe$.

\begin{theorem} \label{t:refinedcount'} Let $\sigma$ be any connected skew shape, and fix sequences $\bfc$ and $\bfe$.  Then 
$$|\Semi_{\bfc,\bfe}(\sigma)| = \sum_{\mu \supseteq \sigma}  (-1)^{|A(\mu/\sigma)|} \cdot a_{\sigma,\mu,\bfe} \cdot |\Semi_{\bfc, \mathbf{0}} (\mu)|,$$
where $a_{\sigma,\mu,\bfe}$ are the nonnegative integers defined in Theorem~\ref{t:refinedcount}.
\end{theorem}

Thus Theorems ~\ref{t:refinedcount} and~\ref{t:refinedcount'} are equivalent.

\remnow{The change from $B(\mu/\sigma)$ in Theorem~\ref{t:refinedcount} to $A(\mu/\sigma)$  in Theorem~\ref{t:refinedcount'} is not accidental; it arises from the definition of $RG$ as a {\em signed} generating function for set-valued semistandard tableaux.}

Then, by specializing to $\bfw = 1$ in Theorem~\ref{t:refinedcount}, we obtain Theorem~\ref{thm:comb}.

% Removed since it is redundent with the theorem in the intro.
%the following expansion of $G_\sigma$ in terms of skew Schur functions.

%\thmnow{\label{t:maincount} 
%For any skew shape $\sigma$,

%$$G_\sigma = \sum_{\mu\supseteq \sigma} (-1)^{|B(\mu/\sigma)|} a_{\sigma,\mu} \cdot s_\mu$$
%where the $a_{\sigma, \mu}$ are nonnegative integers. In fact $a_{\sigma,\mu}$ is the product of the following two integers:
%\enumnow{
%\item the number of row-weakly-bounded semistandard tableaux of shape of $A(\mu/\sigma)$, and
%\item the number of row-bounded, reverse row-strict tableaux of shape $B(\mu/\sigma)$ .
%}

%}

\begin{remark} \label{rem:connection-to-lenart}
Consider the row-bounded, reverse row-strict tableaux of shape $B(\mu/\sigma)$, as in (2) above. There is a bijection between this set and the set of row-bounded, row- and column- strictly-decreasing tableaux of shape $B(\mu/\sigma)$, obtained by replacing label $T(i,j)$ with $i-T(i,j)$.  Therefore, when $\sigma$ is a straight shape whose highest row is in row $1$, Theorem~\ref{thm:comb} reduces to \cite[Theorem 2.2]{lenart-combinatorial}.  In particular, $A(\mu/\sigma)$ is always empty in this case.  

We also note that when $N=|\sigma|+1$, Theorem~\ref{thm:comb} specialized to the monomial $x_1\cdots x_N$ is equivalent to \cite[Theorem 2.8]{clpt}.
Moreover a proof using a row insertion algorithm in the special case that $\sigma$ is a straight shape and $N=|\sigma|+1$  is presented in \cite[Proposition 3.9]{reiner-tenner-yong}.
\end{remark}

\remnow{\label{rem:connection-to-act}
The determinantal formula of \cite{anderson-chen-tarasca-k-classes} can also be expanded as a similar sum involving enumeration of standard young tableaux on larger skew shapes (see \cite[Theorem C]{anderson-chen-tarasca-k-classes}).
Thus, Theorem~\ref{thm:comb} establishes in a purely combinatorial manner that the determinantal formula in  \cite{anderson-chen-tarasca-k-classes} is equal to the number of set-valued tableaux.
}

\remnow{\label{rem:darij}
D.~Grinberg has pointed out the row-refined skew stable Grothendieck polynomials $RG_\sigma(\mathbf{x};\mathbf{w})$ in Definition~\ref{def:rg}, restricted to straight shapes $\sigma$, are Hall-dual to the power series $\tilde{g}_{\mu}(\mathbf{x};\mathbf{w})$ defined in \cite{galashin-grinberg-liu-refined}.  In other words, 
$$\langle RG_\sigma(\mathbf{x};\mathbf{w}), \tilde{g}_\mu(\mathbf{x};\mathbf{w})\rangle = \delta_{\sigma,\mu}$$
for straight shapes $\sigma$ and $\mu$. Here $\langle\cdot,\cdot\rangle$ denotes the Hall inner product on symmetric functions in variables $\mathbf{x}$, treating the $w_i$ as scalars. }

Now we prove Theorem~\ref{t:refinedcount'} using a new generalized row insertion algorithm.  This proof occupies the rest of the section.  This algorithm extends the set-valued insertion algorithm in \cite{bandlow-morse-combinatorial} to the case of skew shapes. 

\defnow{\label{def:rsk}(RSK row insertion) First, recall the row insertion operation, the atomic operation of the RSK algorithm \cite[\S7.11]{ec2} (we present a very slightly more general version). Suppose $\sigma$ is a skew or straight shape and $T$ is a semistandard tableau of shape $\sigma$.  Given $k\in \NN$ and $i$, the operation $T\leftarrow_i k$ inserts $k$ in the leftmost box of row $i$ labeled $j>k$, or a new box at the right end of the $i^\text{th}$ row if no box in that row is labeled $>k$ (in the case where there are not yet any boxes in that row, the new box is placed directly below the leftmost entry in the previous row).  In the latter case the operation terminates. In the former, we insert $j$ into the $(i+1)^\text{st}$ row of $\sigma$ in the same manner, and repeat down the rows of $\sigma$.  The {\em insertion path} is the sequence of boxes $b_{i,j_1},b_{i+1,j_2},\ldots$ in which insertions occurred; one can check that $j_1 \ge j_2 \ge \cdots$ \cite[Lemma 7.11.2]{ec2}.
}

In particular, row insertion inputs a semistandard tableau of shape $\sigma$ and outputs a semistandard tableau of shape $\sigma'$ obtained by adding one box to $\sigma$.

\begin{remark}
Notice that row insertion may be applied without changes to set-valued tableaux in the following situation.  Suppose $T$ is a set-valued semistandard tableau of shape $\sigma$.  Suppose $k$ is a label in a box $b$ with at least one other label; let $i$ index the row containing $b$.  Suppose further that every box in row $i+1,i+2,\ldots$ is labeled with a singleton set.  Then one may define the operation $T\leftarrow_i k$ as before, deleting $k$ from box $b$ and row-inserting it in the next row, and repeating.  Simply put, the row insertion path does not traverse any box with more than one label in this case.

This observation allows for the next algorithm.
\end{remark}

\begin{algorithm}\label{algorithm}
The {\em skew set-valued} row insertion algorithm, for a connected skew shape $\sigma$, is as follows.  The  input is
\enumnow{
\item a skew shape $\lambda \supseteq \sigma$ with $B(\lambda/\sigma)=\emptyset$,
\item $T'$ a reverse-row-strict, row-weakly-bounded tableau on $A(\lambda/\sigma)$, and
\item $T\in \Semi_{\bfc,\bfe-c(T')}(\lambda)$.
}
The output will be:
\enumnow{
\item a skew shape $\mu\supseteq \sigma$ with $A(\mu/\sigma)=\emptyset$,
\item $T''$ a reverse-row-strict, row-bounded tableau on $B(\mu/\sigma)$ with $c(T'')=\bfe$, and
\item $\wt{T} \in \Semi_{\bfc,\bfzr}(\mu)$.
}
\end{algorithm}

The algorithm proceeds as follows.  Let $r$ be the number of rows of $\sigma$. For each $k=r,\ldots,1$ (in descending order), we will do two ``sweeps'' of $\sigma$. First we sweep out all labels in row $k$ that are not the minimum in their box, via row-inserting them downward. Then we sweep out all labels in all (singly-labeled) boxes $b$ for which $T'(b) = k$, again via row insertion.  These boxes need not be in row $k$.  In the auxiliary labeling $T''$, the newly created boxes are labeled $k$, and properties of row insertion will imply that at most one box in each column of $T''$ is labeled $k$. An example is given in Example~\ref{ex:RSK}.  

Now we describe the algorithm more precisely. For $k=r,\ldots,1$, proceed as follows.  First, let $m$ be the maximum label in the rightmost box of row $k$ that has multiple labels.  Delete $m$ and insert $m$ into the leftmost box of row $k+1$ labeled $m_2>m$, or a new box at the right end of the ${k+1}^\text{st}$ row if no box in that row is labeled $>m$.  In the latter case the operation terminates; the new box is labeled $k$ in the auxiliary filling $T''$. In the former, we insert $m_2$ into the $(k+2)^\text{nd}$ row of $\sigma$ in the same manner, and repeat down the rows of $\sigma$.  This is the familiar row-insertion operation of Definition~\ref{def:rsk}.  The {\em insertion path} is the sequence of boxes $b_0 = (k,j_0),b_1=(k\!+\!1,j_1),\ldots$ in which insertions occurred; one can check that $j_0 \ge j_1 \ge \cdots$ (\cite[Lemma 7.11.2]{ec2}).  Repeat row-insertion on the maximum label in the rightmost non-singly valued box in row $k$, until that row has only singly-valued boxes.  

The second part of step $k$ is as follows.  Since $T'$ is row-weakly-bounded and reverse row-strict, it follows that there is at most one box $(i,j)\in A(\lambda/\sigma)$ in each row such that $T'(i,j)= k$; furthermore $i\ge k$ if so, so that $T(i,j)$ must consist of a single label.  So for each such box $(i,j)$, taken in order with $i$ increasing, delete the box and row-insert the label $T(i,j)$ it starting in row $i+1$.  When the operation terminates, the new box is labeled $k$ in the auxiliary filling $T''$.  We note for later use that in this stage, every box labeled $k$ in $T''$ is the leftmost of its row, since all boxes labelled $>k$ have already been removed.

\exnow{\label{ex:RSK}Let $$\sigma = \tiny \young(::\ ,::\ ,:\ \ ,\ \ ) \qquad \lambda = \tiny \young(:\ \ ,\ \ \ ,\ \ \ ,\ \ ) \qquad T' = \tiny \young(:1,21,1) \qquad T=
\ytableausetup{boxsize=.75cm} 
\begin{ytableau}
 \none & 2 &3\\
 1,4 & 6 & 8 \\
 5,7 & 9 & 10,\blue{13} \\
 11 & 12
 \end{ytableau}.
$$
The algorithm gives

$$\ytableausetup{boxsize=.5cm} 
 \begin{ytableau}
 \none & 2 &3\\
 1,4 & 6 & 8 \\
 5,\blue{7} & 9 & 10 \\
 11 & 12 & \red{13}
 \end{ytableau}
 \qquad
  \begin{ytableau}
 \none & 2 &3\\
 1,\blue{4} & 6 & 8 \\
 5 & 9 & 10 \\
 7 & 12 & \red{13}\\
 \red{11}
 \end{ytableau}
 \qquad
  \begin{ytableau}
 \none & 2 &3\\
 \blue{1} & 6 & 8 \\
 4 & 9 & 10 \\
 5 & 12 & \red{13} \\
 \red{7}\\
 \red{11}
 \end{ytableau}
 \qquad
  \begin{ytableau}
 \none & \blue{2} &3\\
 \none & 6 & 8 \\
 1 & 9 & 10 \\
 4 & 12 & \red{13} \\
 \red{5}\\
 \red{7}\\
 \red{11}
 \end{ytableau}
\qquad
  \begin{ytableau}
 \none & \none &3\\
 \none & \blue{2} & 8 \\
 1 & 6 & 10 \\
 4 & 9 & \red{13} \\
 \red{5} & \red{12}\\
 \red{7}\\
 \red{11}
 \end{ytableau}
 \qquad
  \begin{ytableau}
 \none & \none &3\\
 \none & \none & 8 \\
 \blue{1} & 2 & 10 \\
 4 & 6 & \red{13} \\
 \red{5} & \red{9}\\
 \red{7}& \red{12}\\
 \red{11}
 \end{ytableau}
\qquad
  \begin{ytableau}
 \none & \none &3\\
 \none & \none & 8 \\
 \none & 2 & 10 \\
 1 & 6 & \red{13} \\
 \red{4}& \red{9}\\
 \red{5} & \red{12}\\
 \red{7}\\
 \red{11}
 \end{ytableau}
 $$
and the auxiliary tableau $T''$ is $$\young(::3,31,21,2,1)$$

}

\begin{lemma}\label{lem:asclaimed}
The output of Algorithm~\ref{algorithm} takes the claimed form.
\end{lemma}

\begin{proof}
We remark that the process in Algorithm~\ref{algorithm} preserves the property that {\em every box in row $k+1$ and below has exactly one label in it}, so the row-insertion is always well-defined.  The process also clearly preserves the content of the tableau $T$. Thus iterating the described two-step process for $k=r,\ldots,1$ produces the output data $\mu, T'',$ and $\wt T$, with $c(\wt{T}) = c(T)$. Furthermore $T''$ is row-bounded since $T'$ was row-weakly-bounded. To conclude that the output is as claimed, the only thing left to show is that the labeling $T''$ of $B(\mu/\sigma)$ is reverse row-strict.

 Indeed, since the rows are processed in the order $r,\ldots,1$ in Algorithm~\ref{algorithm}, it is enough to show that for a fixed $k\in\{1,\ldots,r\}$ that no two boxes labelled $k$ in $T''$ lie in the same row.  This follows from the standard fact that row-insertion paths move weakly to the left. Precisely: Suppose $m$ and $m'$ are labels that are processed consecutively in step $k$.  Let $b_0,b_1,\ldots b_M$ be the insertion path of $m$.  By assumption, after $m$ is inserted, every box $b_i$ except possibly $b_0$ is still singly labeled, and $\max(T(b_0)) < T(b_1) < \ldots < T(b_M)$. 

Furthermore, we claim that the label $m'$ is on or to the left of the insertion path of $m$.  Indeed, if $m'$ is also in row $k$, then this is clear since $m'<m$; otherwise, we simply note that $m'$ is in the leftmost box of its row, so the claim is also clear.  Finally, row-insertion of $m'$ preserves the property of being weakly left of the insertion path of $m$.  So the insertion path of $m'$ cannot end to the right of that of $m'$; thus it ends below that of $m'$. This concludes the proof of the lemma.
\end{proof}

Now we show that all possible outputs are attained bijectively by the algorithm.

\begin{prop} \label{p:main}
For any connected skew shape $\sigma$ and any $\bfc$ and $\bfe$, Algorithm~\ref{algorithm} produces a bijection between choices of 
\enumnow{
\item a skew shape $\lambda \supseteq \sigma$ with $B(\lambda/\sigma)=\emptyset$,
\item $T'$ a reverse-row-strict, row-weakly-bounded tableau on $A(\lambda/\sigma)$, and
\item $T\in \Semi_{\bfc,\bfe-c(T')}(\lambda)$;
}
and choices of
\enumnow{
\item a skew shape $\mu\supseteq \sigma$ with $A(\mu/\sigma)=\emptyset$,
\item $T''$ a reverse-row-strict, row-bounded tableau on $B(\mu/\sigma)$ with $c(T'')=\bfe$, and
\item $\wt{T} \in \Semi_{\bfc,\bfzr}(\mu)$.
}
Therefore, 

\begin{equation}\label{eq:ss}
\sum_{(\lambda, T')} |\Semi_{\bfc, \bfe-c(T')}(\lambda)| = \sum_{(\mu,T'')} |\Semi_{\bfc, \mathbf{0}}(\mu)|
\end{equation}
where
\begin{itemize}
\item the left hand sum ranges over all $\lambda\supseteq \sigma$ with $B(\lambda/\sigma)=\emptyset$, together with a reverse row-strict, row-weakly-bounded labeling $T'$ of $A(\lambda/\sigma)$, and
\item the right hand sum ranges over all $\mu\supseteq \sigma$ with $A(\mu/\sigma) = \emptyset$, together with a reverse row-strict, row-bounded labeling $T''$ of $B(\mu/\sigma)$.
\end{itemize}
\end{prop}

\begin{proof}
 The skew set-valued row-insertion algorithm in Algorithm~\ref{algorithm} constructs a map
\begin{equation}\label{eq:ssmap}
\coprod_{(\lambda, T')} \Semi_{\bfc, \bfe-c(T')}(\lambda) \xrightarrow{\text{RSK}_\sigma}\coprod_{(\mu,T'')} \Semi_{\bfc, \mathbf{0}}(\mu),
\end{equation}
where the conditions on $\lambda,\mu,T',$ and $T''$ are as in the statement of the proposition. We claim this map is a bijection, and it suffices to provide an inverse.  The inverse may be described algorithmically as follows.  Given $\mu, T'',$ and $\wt T$ satisfying conditions (1), (2), and (3) described as the output of Algorithm~\ref{algorithm}, perform the following procedure for $k=1,\ldots, r$.  Consider the boxes of $B(\mu/\sigma)$ labelled $k$ in $T''$, in order from highest to lowest row number (i.e.~lowest to highest on the page).  For each such box $b$, delete $b$ and inverse-row-insert its label $m$ upwards, stopping if it reaches row $k$.  If the label $m$ lands in a new box $b'$, necessarily in row $\ge k$, then set $T'(b) = k$.  An example is given in Example~\ref{ex:RSK}, read in reverse.

The resulting tableau $T'$ is reverse-row-strict by an argument  analogous to Lemma~\ref{lem:asclaimed}.  So the result of this procedure is the data $\lambda,T',$ and $T$ satisfying the conditions (1), (2), and (3) described as the input  of Algorithm~\ref{algorithm}. 
Now it is evident that the procedure described is in fact inverse to the RSK map in Algorithm~\ref{algorithm}, since each upwards insertion operation is inverse to row insertion, and it processes boxes in the reverse order.
\end{proof}

Now we state %the following 
a Definition and 
Lemma, which will be used to prove Theorem~\ref{t:refinedcount}. We will postpone its proof until after the the proof of Theorem~\ref{t:refinedcount}.  The need for Definition~\ref{def:acyclic} as a hypothesis for Lemma~\ref{lem:Gseq} was pointed out to us by D.~Grinberg.

\defnow{\label{def:acyclic}
Let $P$ be any finite poset, with its set of cover relations $C = \{(x,y)\in P\times P: x\lessdot y\}$ partitioned into two disjoint sets $C=G\sqcup B$ (called {\em good} and {\em bad}, colloquially).  We will say that the partition is {\em acyclic} if the directed graph on the Hasse diagram of $P$ obtained by orienting all good cover relations up and all bad cover relations down is acyclic.
}
\exnow{The partition $G\sqcup B$ of the cover relations in the Hasse diagram is drawn below is {\em not} acyclic.
$$
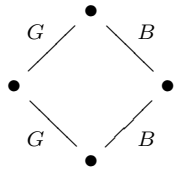
\xymatrix@R=6mm@C=6mm{&\bullet \ar@{-}[dl]_G\ar@{-}[dr]^B&\\ \bullet \ar@{-}[dr]_G && \bullet \ar@{-}[dl]^B\\ &\bullet&}$$
}
\lemnow{\label{lem:Gseq}
Let $P$ be any finite poset, 
with its set of cover relations $C = \{(x,y)\in P\times P: x\lessdot y\}$.  Let $C=G\sqcup B$ be an acyclic partition of $C$, with elements of $G$ and $B$ called {\em good} and {\em bad} throughout.  
Say that an increasing sequence
$$\mathcal{I} = (\emptyset = I_0 \subsetneq \cdots \subsetneq I_\ell = P)$$
of order ideals $I_i$ is a {\em $G$-sequence} if for every $i=1,\ldots,\ell$, the only cover relations within $I_i \setminus I_{i-1}$ are in $G$.  Precisely: if $x,y\in I_i \setminus I_{i-1}$ and $x\lessdot y$ then $(x,y)\in G$.  The length of such a $G$-sequence $\mathcal{I}$ is defined to be $|\mathcal{I}|=\ell$.
Then

\begin{equation}\label{eq:signedcount}
\sum_\text{$\mathcal{I}$ a $G$-sequence} \! (-1)^{|\mathcal{I}|} = \begin{cases}
(-1)^{|P|} & \text{if }G=\emptyset,\\
0 &\text{otherwise.}
\end{cases}
\end{equation}
}

\exnow{
Suppose $P = P(\lambda)$ is the poset of boxes of a diagram $\lambda$.  If $G=\emptyset$ then the $G$-sequences are in natural correspondence with increasing tableaux of shape $\lambda$ with label set $\{1,\ldots,N\}$ for some $N$.  Lemma~\ref{lem:Gseq} states that counting these increasing tableaux, with sign according to the parity of $N$, is $(-1)^{|P|}$.  For example, if $P=P(\tiny \yng(3,1))$ then the Lemma states that 
$$(-1)^4 = \#\left\{\young(123,4),\young(124,3), \young(134,2) \right\} - \#\left\{\young(123,3) ,\young(123,2)\right\}.$$
}

Postponing the proof of Lemma~\ref{lem:Gseq}, we now prove Theorem~\ref{t:refinedcount'}.  

\begin{proof}[Proof of Theorem~\ref{t:refinedcount'}]  
%The proof is by induction on $|\bfe|$ with $|\bfe|=0$ being clear.  
We fix $\sigma$ and sequences $\bfc$ and $\bfe$, with $|\bfc| = |\sigma| + |\bfe|$; otherwise the statement is trivial.    Now isolating the term  $|\Semi_{\bfc, \bfe}(\sigma)|$ on the left of Equation~\eqref{eq:ss}, we have
\begin{equation}\label{eq:notyetrecursed}
|\Semi_{\bfc, \bfe}(\sigma)| = \sum_{(\mu,T'')} |\Semi_{\bfc, \mathbf{0}}(\mu)| - \sum_{(\lambda, T'): \lambda\supsetneq \mu } |\Semi_{\bfc, \bfe-c(T')}(\lambda)|
\end{equation}
where the conditions on $(\mu,T'')$ and $(\lambda,T')$ are as stated in Proposition~\ref{p:main}.  Now we may use Proposition~\ref{p:main} inductively to expand each of the terms $|\Semi_{\bfc, \bfe-c(T')}(\lambda)|$ in the second sum of Equation~\eqref{eq:notyetrecursed}.  We obtain
\begin{equation}\label{eq:nowrecursed}
|\Semi_{\bfc, \bfe}(\sigma)| = \sum_{(\mu,T',T'')} b_{\sigma,\mu,T',T''} \cdot |\Semi_{\bfc, \mathbf{0}}(\mu)|,\end{equation}
 for some coefficients $b$ which we will soon study. Here
\itemnow{
\item $\mu\supseteq \sigma$ is a skew shape,
\item $T'$ is {\em any} row-weakly-bounded filling of $A(\mu/\sigma)$, 
\item $T''$ is a reverse-row-strict, row-bounded filling of $B(\mu/\sigma)$,
}
such that
$$c(T') + c(T'') = \bfe.$$

To prove Theorem~\ref{t:refinedcount'} it is enough to show that the coefficients on the right hand side are given by
$$b_{\sigma,\mu,T',T''} = \begin{cases}
(-1)^{|A(\mu/\sigma)|} &\text{ if $T'$ is semistandard}, \\ 
0 & \text{ otherwise.}
\end{cases}$$

Indeed, it follows from the recursive expansion of Equation~\eqref{eq:notyetrecursed} that the coefficient $b_{\sigma,\mu,T',T''}$ depends only on $T'$: it is
the signed count of the number of ways to build $T'$ as a sequence of tableaux
$$\emptyset = T_0 \subsetneq T_1 \subsetneq \cdots \subsetneq T_\ell = T'$$
on a corresponding sequence of skew shapes
$$\emptyset = \lambda_0 \subsetneq \lambda_1 \subsetneq \cdots \subsetneq \lambda_\ell = A(\mu/\sigma)$$
 such that each $\lambda_i$ is obtained from $\lambda_{i-1}$ by adding boxes on the left or above boxes of $\lambda_{i-1}$, and the restriction of $T_i$ to $\lambda_i/\lambda_{i-1}$ is reverse row-strict for each $i$.  By the {\em signed count}, we mean that such a sequence is counted with sign $(-1)^\ell.$   

For example, a filling $T' =$  $$\young(21,22)$$ can be obtained in the following ways, with the following signs:

\begin{center}
\begin{tabular}{cccccr}
$\young(\ \ ,\ \ )$ & $ \young(\ \ ,\ 2)$ & $ \young(\ \ ,22)$ & $\young(\ 1,22)$ & $\young(21,22)$ &$+$\\
\vspace{.1cm}\\
$\young(\ \ ,\ \ )$ & $ \young(\ \ ,\ 2)$ & $ \young(\ 1,\ 2)$ & $\young(\ 1,22)$ & $\young(21,22)$ & $+$\\
\vspace{.1cm}\\
&$\young(\ \ ,\ \ )$ & $ \young(\ \ ,\ 2)$ & $ \young(\ 1,22)$ & $\young(21,22)$ &$-$\\
\vspace{.1cm}\\
&$\young(\ \ ,\ \ ) $ & $ \young(\ \ ,\ 2)$ & $ \young(\ \ ,22)$ & $\young(21,22)$ & $-$
\end{tabular}
\end{center}

\noindent and so $b_{\sigma,\mu,T',T''} = 0$ for this $T'$.

Thus, to compute $b_{\sigma,\mu,T',T''}$ in general, we let $P=P(A(\mu/\sigma))$ be the poset whose elements are boxes of $A(\mu/\sigma)$ and such that $b_1\lessdot b_2$ if and only if box $b_1$ is directly to the right of or directly below $b_2$. Now let $G$ be the subset of cover relations $b_1\lessdot b_2$ of $P$ in which either
\itemnow{
\item $b_1$ is directly to the right of $b_2$ and $T'(b_2) > T'(b_1)$, or
\item $b_1$ is directly below $b_2$ and $T'(b_2) \ge T'(b_1)$.
}
Let $B$ be the set of cover relations not in $G$.  Observe that $G$ and $B$ are an acyclic partition of the cover relations, in the sense of Definition~\ref{def:acyclic}.  Indeed, a cycle in the oriented Hasse diagram, say on boxes $b_0,b_1,\ldots,b_t=b_0$ would correspond to a sequence of inequalities $T(b_0) \le T(b_1) \le \cdots \le T(b_t) = T(b_0)$, and since the boxes $b_0,\ldots,b_t$ occupy more than one row and one column,  at least one (in fact at least two) of those $t$ inequalities are strict, which is clearly impossible.
Then by Lemma~\ref{lem:Gseq} it follows that 
$$b_{\sigma,\mu,T',T''} = \begin{cases}
(-1)^{|A(\mu/\sigma)|} &\text{ if }G=\emptyset, \\ 
0 & \text{ otherwise.}
\end{cases}$$
But $G=\emptyset$ means precisely that $T'$ is semistandard.  
\end{proof}

It remains only to prove Lemma~\ref{lem:Gseq}.
\begin{proof}[Proof of Lemma~\ref{lem:Gseq}] We prove Lemma~\ref{lem:Gseq} by induction on $|P|$, with $P=\emptyset$ being obvious.  

Write $J(P)$ for the set of order ideals of $P$. 
We break up \eqref{eq:signedcount} according to the first order ideal $I_1$ and proceed inductively on $P\setminus I_1$.  Start with the equality
\begin{equation}\label{eq:first}
\sum_\text{$\mathcal{I}$ a $G$-sequence} \! (-1)^{|\mathcal{I}|} = \sum_{\emptyset\ne A \in J(P)}  \,\,\sum_{\stackrel{\text{$\mathcal{I}$ a $G$-sequence}}{I_1=A}} (-1)^{|\mathcal{I}|}.
\end{equation}
Notice that for any order ideal $A$, the partition on the cover relations of $P\setminus A$ induced by $G\sqcup B$ is again acyclic.  Then by induction, the nonzero contributions to the right hand side of~\eqref{eq:first}  come from nonempty order ideals $A$ in which 
\itemnow{
\item all cover relations inside $A$ are good,
\item all cover relations inside $P\setminus A$ are bad.
}
Let $\mathcal{A}$ be the set of nonempty order ideals of $P$ satisfying these conditions.  Then using~\eqref{eq:signedcount} inductively,~\eqref{eq:first} becomes
\begin{equation}\label{eq:second}
\sum_\text{$\mathcal{I}$ a $G$-sequence} \! (-1)^{|\mathcal{I}|} = \sum_{A\in \mathcal{A}}  \,\,(-1)\cdot (-1)^{|P\setminus A|}.
\end{equation}

It remains to identify $\mathcal{A}$ in terms of $P$ and $G$, which we do as follows. Let $Y$ be the maximal up-closed subset of $P$ such that all cover relations within $Y$ are bad.  Note that $Y$ is uniquely defined, since if $Y_1$ and $Y_2$ are up-closed subsets satisfying that condition, then $Y_1\cup Y_2$ also satisfies the condition.  

Let $X = P\setminus Y$.  Let $Y'\subseteq Y$ be the subset consisting of the minimal elements $y\in Y$ satisfying that if $x\lessdot y$ then $(x,y)\in G$. 
Then we claim 
\begin{claim}\label{c:A}\mbox{}\enumnow{
\item If $X = \emptyset$ then $\mathcal{A} = \{I \subseteq \min(P): I\ne \emptyset\}$, where $\min(P)$ denotes the minimal elements of $P$.
\item If $X\ne \emptyset$ and some cover relation within $X$ is in $B$, then $\mathcal{A}=\emptyset$.
\item If $X\ne \emptyset$ and all cover relations within $X$ are in $G$, then
$$\mathcal{A} = \{X\cup I: I\subseteq Y'\}.$$
}
\end{claim}

\begin{proof}[Proof of Claim~\ref{c:A}.] If $X=\emptyset$ then $G=\emptyset$, and $\mathcal{A}$ then consists of all nonempty order ideals with no cover relations within them. So part (1) follows.

Suppose $X\ne \emptyset$.  Suppose $A\in \mathcal{A}$. Now for each maximal element $x\in X$, there is some $y\in Y$ such that $x\lessdot y$ is good.  So necessarily $x\in A$, since otherwise the covering relation $x\lessdot y$ would lie in $P\setminus A$.  So $A$ contains all maximal elements of $X$; thus $A\supseteq X$.   So if some cover relation within $X$ is in $B$, then $\mathcal{A}=\emptyset$, proving part (2).  

Otherwise, we see that $X\in\mathcal{A}$.  Furthermore, if $A\in\mathcal{A}$ then $A\cap Y$ must be an antichain in $Y$; otherwise $A$ contains a bad cover relation.  So $A\cap Y\subseteq \min(Y)$.  And if $y\in A\cap Y$, then any cover relation $x\lessdot y$ must be good.  We conclude that $\mathcal{A}  \subseteq \{X\cup I: I\subseteq Y'\}$;  the reverse containment also clearly follows.
\end{proof}

Now from Claim~\ref{c:A}, the rest of the proof of Lemma~\ref{lem:Gseq} can be deduced from~\eqref{eq:second} by using the obvious identity 
$\sum_{S\subseteq T} (-1)^{|S|} = 0$
for nonempty finite sets $T$.  Explicitly, in the case \ref{c:A}(1), Equation~\eqref{eq:second} becomes 
$$\sum_{\emptyset\ne A \subseteq \min(P)} \,\,(-1)\cdot(-1)^{|P\setminus A|} = (-1)^{|P|}.$$
In the case \ref{c:A}(2), Equation~\eqref{eq:second} is the empty sum.  
In the case \ref{c:A}(3), Equation~\eqref{eq:second} becomes
$$\sum_{I\subseteq Y'} \,\,(-1)\cdot(-1)^{|P\setminus(X \cup I)|} = (-1)^{|P|+1}\sum_{I\subseteq Y'} (-1)^{|I|},$$
which is $0$, as desired, provided that $Y'\ne \emptyset$. It remains only to show that $Y'=\emptyset$ is not possible.

If $Y'=\emptyset$ then by definition of $Y'$, every minimal element $y\in Y$ covers some $x\in X$ such that $x\lessdot y$ is bad. And every maximal element $x\in X$ is covered by some $y \in Y$ such that $x\lessdot y$ is good.  Therefore in the orientation of the Hasse diagram of $P$ described in Definition~\ref{def:acyclic}, {\em every} element in $X$ sends out an upwards edge, and {\em every} element in $Y$ sends out a downwards edge. Therefore the oriented Hasse diagram has no sink and cannot be acyclic, contradicting the hypotheses.

\end{proof}

\section{An inverse formula}

We give an analogous linear expansion of skew Schur functions into skew stable Grothendieck polynomials. This formula generalizes \cite[Theorem 2.7]{lenart-combinatorial}, which pertains to straight shapes, and visibly specializes to that result when $\sigma$ is a straight shape. %When $\sigma$ is a non-skew shape then~\eqref{eq:gsigma-conj} immediately specializes to that theorem.

\begin{theorem}\label{t:back}
For any connected skew shape $\sigma$, %write $\hat{\sigma}$ for the maximal straight shape obtained by adding all boxes that are both above and left of boxes of $\sigma$ (i.e.~``filling in the upper left shadow of $\sigma$'') as well as adding at most $i$ boxes to the right of row $i+1$.
\begin{equation}\label{eq:gsigma-conj}
s_\sigma = \sum_{\mu\supseteq \sigma} (-1)^{|A(\mu/\sigma)|} b_{\sigma,\mu} \cdot G_\mu
\end{equation}
where the $b_{\sigma, \mu}$ is the product of the following two numbers:
\enumnow{
\item the number of row-weakly-bounded, reverse-row-strict tableaux of shape $A( \mu/\sigma)$, and 
\item the number of row-bounded, semistandard tableaux of shape $B(\mu/\sigma)$.
}
\end{theorem}

\begin{proof}
Fix $\sigma$. By Theorem~\ref{thm:comb}, the right hand side of~\eqref{eq:gsigma-conj} is
\begin{equation}\label{eq:a-b}
\sum_{\mu\supseteq \sigma} \sum_{\lambda\supseteq \mu} (-1)^{|A(\mu/\sigma)|+|B(\lambda/\mu)|} a_{\lambda,\mu} \cdot b_{\mu,\sigma} \cdot s_\lambda
%= & \sum_{\mu\supseteq \sigma} \sum_{\lambda\supseteq \mu} (-1)^{|A(\mu/\sigma)|+|B(\lambda/\mu)|} c_{\lambda,\sigma}\cdot s_\lambda
\end{equation}
And $a_{\lambda,\mu} \cdot b_{\mu,\sigma}$ is the product of the sizes of the following two sets:
\begin{enumerate}
\item the row-weakly-bounded tableaux $T'$ on $A(\lambda/\sigma)$ which are semistandard on $A(\lambda/\mu)$ and reverse-row-strict on $A(\mu/\sigma)$, and
\item the row-bounded tableaux $T''$ on $B(\lambda/\sigma)$ which are reverse-row-strict on $B(\lambda/\mu)$ and semistandard on $B(\mu/\sigma)$.
\end{enumerate} 

Now for a fixed $\lambda$, a row-weakly-bounded filling $T'$ of $A(\lambda/\sigma)$, and a row-bounded filling $T''$ of $B(\lambda/\sigma)$, the contribution of the triple $(\lambda,T',T'')$ to the sum~\eqref{eq:a-b} is
\begin{equation}\label{eq:should-vanish}
\sum_{\substack{\mu\text{ with } \sigma\subseteq \mu\subseteq \lambda \\T',T''\text{ satisfy }(1),(2)}}
 (-1)^{|A(\mu/\sigma)|+|B(\lambda/\mu)|}\cdot  s_\lambda
 \end{equation}

Now if $\lambda = \sigma$ then~\eqref{eq:should-vanish} equals $s_\sigma$ vacuously.  Otherwise, we show that~\eqref{eq:should-vanish} vanishes. In words, the coefficient of $s_\lambda$ in \eqref{eq:should-vanish} is a signed count of ways to divide the two shapes $A(\lambda/\sigma)$ and $B(\lambda/\sigma)$, such that in each, the upper-left is semistandard and the lower-right is reverse-row-strict.  Then the following lemma, applied to $A(\lambda/\sigma)$ and $B(\lambda/\sigma)$, finishes the proof of Theorem~\ref{t:back}. 
\end{proof}

\begin{lemma}\label{lem:ss-rrs}
Let $\mu$ be any nonempty skew shape. A {\em division} of $\mu$ is a partition $\mu = \mu_S \sqcup \mu_R$ into two skew shapes such that $\mu_S$ is an order ideal of $\mu$ considered as a poset (Definition~\ref{def:young}).  In other words, no box of $\mu_R$ is north/west of any box of $\mu_S$.
Suppose $T$ is a (non-set-valued) tableau on $\mu$.  Say that $(\mu_S,\mu_R)$ is {\em allowed} by $T$ if $T$ is semistandard on $\mu_S$ and reverse-row-strict on $\mu_R$.  Then
$$\sum_{(\mu_S,\mu_R) \text{ allowed by }T} (-1)^{|\mu_S|} = 0.$$
\end{lemma}

\exnow{\label{ex:strip}If $T = (a_1,\ldots,a_n)$ is a tableau on a single row of length $n$, then the above sum is empty unless $a_1 \le \cdots \le a_M > \cdots > a_n$ for a unique index $M$. In this case, there are two allowable divisions: where $\mu_S$ is either the first $M$ or the first $M-1$ boxes.}

\begin{proof}[Proof of Lemma~\ref{lem:ss-rrs}]
Each cover relation $b_1 \lessdot b_2$ in $\mu$ may be labelled $S$ or $R$ according to whether the restriction of $T$ to the $2$-box shape $\{b_1,b_2\}$ is semistandard or is reverse-row-strict. In the first case, write $S$ in box $b_1$; in the second case, write $R$ in box $b_2$. In this way, fill the boxes of $\mu$ using the alphabet $\{\emptyset, S, R, SR\}$.  (In Example~\ref{ex:strip}, the first $M-1$ boxes are $S$, the next box is empty, and the remaining are $R$.)  

Now if there are any allowable divisions at all, then no box is labelled $SR$, the $S$ boxes must be an order ideal, the boxes containing $R$ must be the complement of an order ideal, and the empty boxes form an antichain in between.  Then $(\mu_S,\mu_R)$ is allowable if and only if $\mu_S$ contains all $S$ boxes and $\mu_R$ contains all $R$ boxes.  Then to prove the lemma, it is enough to show that there exists at least one empty box.  Consider, among all boxes with maximum number in $T$, an upper-rightmost one.  That box has empty label.
\end{proof}

\bigskip

\noindent {\bf Acknowledgments.} We thank Vic Reiner, Bridget Tenner and Alexander Yong for generously answering some questions over email about their work \cite{reiner-tenner-yong}. We also thank Jennifer Morse, Travis Scrimshaw, and an anonymous referee for helpful contextual remarks and references.   We are grateful to Darij Grinberg for correcting an earlier error in the formulation of Lemma~\ref{lem:Gseq}, and for other detailed comments.
MC was supported by an NSA Young Investigators Grant and NSF DMS-1701924.

\bibliographystyle{amsalpha}
\bibliography{../my}

\end{document}